\documentclass[10pt]{article}
\font\smallit=cmti10
\font\smalltt=cmtt10

\usepackage{amssymb,latexsym,amsmath,epsfig,amsthm} 
\usepackage[T1]{fontenc}

\makeatletter

\renewcommand\section{\@startsection {section}{1}{\z@}
{-30pt \@plus -1ex \@minus -.2ex}
{2.3ex \@plus.2ex}
{\normalfont\normalsize\bfseries\boldmath}}

\renewcommand\subsection{\@startsection{subsection}{2}{\z@}
{-3.25ex\@plus -1ex \@minus -.2ex}
{1.5ex \@plus .2ex}
{\normalfont\normalsize\bfseries\boldmath}}

\renewcommand{\@seccntformat}[1]{\csname the#1\endcsname. }

\makeatother

\newtheorem{theorem}{Theorem}
\newtheorem{lemma}{Lemma}

\theoremstyle{definition}

\newtheorem*{Ko}{Korselt's Criterion}
\newtheorem{conjtd}{Conjecture}

\newtheorem{ques}{Question}

\begin{document}

\begin{center}
\uppercase{\bf \boldmath $(a,a)$-Carmichael numbers and greatest common divisors of $p-a$}
\vskip 20pt
{\bf Thomas Wright}\\
{\smallit Department of Mathematics, Wofford College, 429 N. Church St., Spartanburg, SC, USA}\\
{\tt wrighttj@wofford.edu}
\end{center}
\vskip 20pt
\centerline{\smallit Received: , Revised: , Accepted: , Published: } 
\vskip 30pt

\centerline{\bf Abstract}
\noindent
Define an $(a,a)${\it-Carmichael number} to be a squarefree natural number $n$ such that $p\mid n$ implies $p-a\mid n-a$.  For such a number $n$ with prime factors $p_1,\cdots,p_m$, define
$$K=GCD[p_1-a,\cdots,p_m-a],$$
and let $C_\nu(X,a)$ denote the number of $(a,a)$-Carmichael numbers up to $X$ such that $K=\nu$.   Assuming a strong conjecture on the first prime in an arithmetic progression, we prove that for any integer $a$ and for any natural number $\nu$ with $(\nu,a)=1$ and $a$ and $\nu$ having opposite parity,
$$C_\nu(X,a)\geq X^{1-(2+o(1))\frac{\log\log\log \log X}{\log\log\log X}}.$$
This is a departure from many traditional constructions of Carmichael numbers, which generally require $K$ to grow along with $n$.
\pagestyle{myheadings}
\markright{\smalltt INTEGERS: 26 (2026)\hfill}
\thispagestyle{empty}
\baselineskip=12.875pt
\vskip 30pt

\section{Introduction}

A {\it Carmichael number} is a composite integer $n$ such that
$$a^n\equiv a\pmod n$$
for every integer $a$.

There is a well-known necessary and sufficient condition for Carmichael numbers, which Korselt discovered in 1899 \cite{Ko}.


\begin{Ko} \textit{A positive composite integer $n$ is a Carmichael number if and only if $n$ is squarefree and $p-1\mid n-1$.}\end{Ko}

While the first Carmichael numbers were discovered over a century ago (see \cite{Ca} and \cite{Si}), a proof that the set of Carmichael numbers is infinite did not appear until 1994 \cite{AGP}.  Following that proof, number theorists raised a number of further questions about Carmichael numbers and related constructs, including the following.

\begin{ques}
For any odd prime $Q$, are there infinitely many Carmichael numbers $n$ for which $Q\mid n$?
\end{ques}

\begin{ques}
For any $a\neq 0$, are there infinitely many numbers $n$ such that, for all primes $p\mid n$, $p-a\mid n-a$?
\end{ques}

\begin{ques}
Is it true that the number of Carmichael numbers up to $X$ is $\gg X^{1-o(1)}$?
\end{ques}

The first question was recently answered in the affirmative by Larsen \cite{LarAP}.  The second was answered in the affirmative by the current author, conditionally in the case of $a\neq \pm 1$ \cite{WrV} and unconditionally in the case of $a=-1$ \cite{WrE}, while the case of $a=1$ is obviously the original \cite{AGP} paper.  The third question was also resolved conditionally by the current author in \cite{WrD}, where it was shown that under the assumption of a strong conjecture on the first prime in an arithmetic progression, 
$$C(X)\geq X^{1-(2+o(1))\frac{\log\log \log \log X}{\log \log\log X}},$$
where $C(X)$ denotes the number of Carmichael numbers up to $X$.

A logical next step would be to combine these questions in some capacity.  A recent paper of Zheng \cite{Zhe} gave the name of $(a,b)${\it -Carmichael numbers} for a number $n$ where $p-a\mid n-b$ for all primes $p\mid n$; we use this nomenclature below.

\begin{ques}
For any prime $Q$ and any $a\neq \pm 1$, are there infinitely many $(a,a)$-Carmichael numbers $n$ for which $Q\mid n$?  Moreover, can one show that the number of $(a,a)$-Carmichael numbers up to $X$ is $\gg X^{1-o(1)}$ as well?
\end{ques}

Unfortunately, this question is quite out of reach at present, even under the assumption of strong conjectures.  However, a potential intermediate step in the resolution of this question comes from the fact that, in \cite{AGP}, the authors required all $p\mid n$ be such that the $p-1$'s share a large common factor $k$.   Importantly, in that paper (and most of the Carmichael papers that have followed it) $k$ must increase as $n$ grows, which means that their method does not allow us to find infinitely many $n$ divisible by a fixed prime $Q$.  In fact, if we define
    \begin{gather*}
K=GCD[p_1-1,\cdots,p_m-1],
\end{gather*}
the results in \cite{AGP}, \cite{WrD}, and \cite{WrE} did not give constructions that would yield infinitely many $n$ for which $K$ was bounded.  

As such, an intermediate step on the way to Question 4 might be the following.  Define $C(X,a)$ to be the number of $(a,a)$-Carmichael numbers up to $X$, and let $C_\nu(X,a)$ denote the number of $(a,a)$-Carmichael numbers up to $X$ for which $K=\nu$.  

\begin{ques}
For any $\nu$ and $a$ of opposite parity, is it true that $C_{\nu}(X,a)\gg X^{1-o(1)}$?
\end{ques}

In his paper, Larsen \cite{LarAP} introduced the idea of using two different $k$'s so as to divorce the size of the GCD from the size of $n$. (The current author also introduced this idea independently in a recent preprint \cite{WrLCM}.)  This allowed him to prove that $C_{\nu}(X,\pm 1)\gg X^{1/168-\epsilon}$ for any even $\nu$. 



Here, we conditionally resolve Question 5 for the general case of $C_{\nu}(X,a)$ by invoking a conjecture of Heath-Brown that has been used many times in the study of Carmichael numbers (see \cite{BP}, \cite{EPT}, \cite{WrD}, and \cite{WrV}).  The full conjecture is as follows.

\begin{conjtd}
For any $A\geq 2$, if $(b,l)=1$ then there exists a prime $p\equiv b\pmod l$ with $$p\ll l\left(\log l\right)^A.$$
\end{conjtd}
It is not expected that this conjecture should hold for $A<2$; indeed, Granville and Pomerance have conjectured that the first prime $p\equiv b\pmod l$ should be $\gg \phi(l) (\log l)^2$ for infinitely many choices of $l$ (see \cite{GPprime}, page 2).  In our result, however, we require only a weakened version of this conjecture.  We state the weakened form of this conjecture in a way that also avoids $\gg$ notation.
\begin{conjtd}
There exists an $A\geq 2$ such that if $l$ is sufficiently large and $(b,l)=1$ then there exists a prime $p\equiv b\pmod l$ with \begin{gather}\label{conjineq}p<l\left(\log  l\right)^A.\end{gather}
\end{conjtd}




In this paper, we prove the following.
\begin{theorem}\label{MainTheorem}
Assume Conjecture 2 holds.  Then for any $\nu$ with $(a,\nu)=1$ and $a$ and $\nu$ having opposite parity,
$$C_\nu(X,a)\geq X^{1-(2+o(1))\frac{\log\log \log \log X}{\log \log\log X}}.$$
\end{theorem}
This is the same lower bound found in \cite{WrD} for the original quantity $C(X,1)$, and it is close to best possible.  Pomerance \cite{Po} proved that
$$C(X,1)\leq X^{1-\frac{\log\log \log X}{2\log\log X}}$$
for sufficiently large $X$, and he subsequently conjectured that
$$C(X,1)\gg X^{1-\frac{\log\log \log X}{\log\log X}}.$$

Thus, while modern construction methods for Carmichael numbers generally require ever-increasing $K$ for ``most" Carmichael numbers, our result would suggest that for any $a\neq 0$, the number of $(a,a)$-Carmichael numbers up to $X$ with bounded $K$ should in fact be relatively close to the size of the set of Carmichael numbers themselves.

\section{Construction Methods}

Thoughout the paper, $p$ will always denote a prime.  Moreover, when we introduce parameters $x$ and $z$, we will assume that these parameters are sufficiently large that all of the statements below are true, since all of the statements hold once the parameters grow sufficiently large.  In fact, for $A$ as defined in Conjecture 2, it will suffice to take $z>e^{e^{100A}}$.  We will also take $z$ to be an integer.

Nearly every modern effort involving Carmichael numbers follows the framework of \cite{AGP}, which depends heavily upon Korselt's criterion; we describe that framework here.  Let $P(y)$ denote the largest prime factor of $y$, and let $\lambda$ denote the Carmichael lambda function.  First, the authors of that paper find a large set of primes $\mathcal Q$ such that for any $q\in \mathcal Q$, $P(q-1)<q^{1-E}$ for some $0<E<1$.  The primes in $\mathcal Q$ are then multiplied together to form
$$L=\prod_{q\in \mathcal Q}q.$$
Because the $q-1$ are smooth relative to $q$, it can be shown that $\lambda(L)$ is small relative to $L$.

Next, the authors define
$$\mathcal P_k=\{p:p=dk+1:d\mid L,\mbox{ }d\leq x^B,\mbox{ and }(L,k)=1\}$$
for a constant $B<1$.  Using results about primes in arithmetic progressions, one can show that there exists a $k_0\geq x^{1-B}$ such that $\mathcal P_{k_0}$ is large if $B<\frac{5}{12}$.  Using a combinatorial theorem of van Emde Boas and Kruyswijk \cite{EK} and Meshulam \cite{Me}, it can then be shown that there are many subsets $\{p_1,\cdots,p_m\}\subset \mathcal P_{k_0}$ such that
$$n=p_1\cdots p_m\equiv 1\pmod L.$$
Clearly, $n$ is also congruent to 1 modulo $k_0$, since $n$ is the product of primes that are 1 mod $k_0$.  So for any $p\mid n$,
$$p-1=dk_0\mid Lk_0\mid n-1.$$
Hence, $n$ is a Carmichael number.

Here, we alter the framework in a way that is somewhat similar to \cite{WrD} and \cite{WrF}.  One of the key ideas in those two papers was to change the way we construct $\mathcal Q$ so as to make $\lambda(L)$ even smaller relative to $L$.  In particular, the method used to construct our primes $p$ can also be used to construct our primes $q$.  Let
$$J=\prod_{\substack{\frac z2\leq r\leq z, \\ r\mbox{ prime
}}}r,$$
and define
$$\mathcal R_j=\{q\mbox{ prime}:q=gj+1,\mbox{ }g\mid J,\mbox{ and }\omega(g)=\lfloor \log z\rfloor\}.$$
Just as before, we can find a $j_0$ for which $\mathcal R_{j_0}$ is relatively large.  Here, the primes $q\in \mathcal R_{j_0}$ are such that $q-1\mid Jj_0$.  Letting $\mathcal Q=\mathcal R_{j_0}$ for some set $\mathcal R_{j_0}$ with many primes, we define $L$ as before and find that $\lambda(L)\mid Jj_0$ as well.  Since this $\lambda(L)$ is very small relative to $L$, we can use much smaller sets of primes $\mathcal P_k$ to find a subset whose product is 1 modulo $L$.

One of the major changes in \cite{LarAP} and \cite{WrLCM} was to create two different (and disjoint) sets $\mathcal Q_1$ and $\mathcal Q_2$.  One can then create an analogous $L_1$ and $L_2$ and prime sets $\mathcal P_{k_1}$ and $\mathcal P_{k_2}$, constructed in such a way that $p_1=d_1k_1\nu+a$ and $(p_1-a,L_2k_2)=1$ for $p_1\in \mathcal P_{k_1}$ and vice-versa for $p_2\in \mathcal P_{k_2}$.  Here, we additionally find a single prime $P$ such that $P=L_1L_2k_1k_2k_3\nu+a$ with $k_3\ll \left(\log\left(L_1L_2k_1k_2\nu\right)\right)^A$.  

Since the $k_i$ are small (as a result of both the construction and the conjecture), it is possible to find sets of primes in $\mathcal P_{k_1}$ that multiply to 1 modulo $k_2k_3L_1L_2$ and sets primes in $\mathcal P_{k_2}$ that multiply to 1 modulo $k_1k_3L_1L_2$.  From the set $\mathcal P_{k_1}$, then, we create a product $n_1$ comprised of primes in this set such that $n_1\equiv 1\pmod{L_1L_2k_1k_2k_3\nu}$; we do the same to find an $n_2$ from $\mathcal P_{k_2}$ such that $n_2\equiv 1\pmod{L_1L_2k_1k_2k_3\nu}$.  Letting $n=Pn_1n_2$, we find that $n$ is an $(a,a)$-Carmichael number with $K=\nu$.

Importantly, we require Conjecture 2 in order to guarantee that $k_1$, $k_2$, and $k_3$ are small.  If, say, $k_2$ were of size $p^{\frac{7}{12}}$ as in \cite{AGP}, or even if $k_2$ were of size $p^\epsilon$ for some small constant $\epsilon$, we would not be able to find enough primes in $\mathcal P_{k_1}$ to guarantee that some subset of them would multiply to 1 modulo $k_2$, and we would not be able to guarantee enough primes in $\mathcal P_{k_2}$ to find $a^{j}\equiv 1\pmod{k_2}$ for some $j\leq \mid \mathcal P_{k_2}\mid $.  One could actually weaken the conjecture somewhat and still prove this result - letting $A=\log\log z$ would still allow the result to be proven - however, we use the requirement that $A$ be a constant to simplify the exposition.

We also note that in most cases below (e.g., lower bounds for $R_{j}$ and $\mathcal P_{k_i}$ and upper bounds for the Carmichael lambda function $\lambda(L)$ and for $L_i$), the bounds here are not close to sharp and can certainly be improved.  However, such improvements would have no effect on the main term of the Main Theorem; indeed, sharpening these bounds to best possible would only affect the $o(1)$-term.  Hence, we content ourselves with the loose bounds below.


\section{Constructing $L_i$}

In \cite{AGP}, the authors find a large set of primes $q$ which will eventually divide $p-1$.  In particular, these $q$'s are chosen such that $q-1$ is fairly smooth; hence, when the authors let $L$ be the product of these $q$'s, they are left with an $L$ for which $\lambda(L)$ is small.  Since we are assuming the conjecture, however, we can find $q$'s for which $q-1$ is very smooth; this will allow us to construct an $L$ for which $\lambda(L)$ is even smaller.  As noted above, this construction was previously used in \cite{WrD} and \cite{WrF}.

First, we construct our $L_i$.  As described above, we let
$$J=\prod_{\substack{\frac z2\leq r\leq z, \\ r\mbox{ prime}}}r,$$
where $z$ is a parameter that is large enough for (\ref{conjineq}) to hold for any $l\geq \frac z2$.

We then consider primes of the form $gj+1$ for $g\mid J$.  Define as before the set
$$\mathcal R_j=\{q\mbox{ prime}:q=gj+1,\mbox{ }g\mid J,\mbox{ and }\omega(g)=\lfloor \log z\rfloor\}.$$
Note that for any prime in $\mathcal R_j$,
\begin{gather}\label{gbound}
g\leq z^{\log z},
\end{gather}
and hence
$$\left(\log  g\right)^A\leq \left(\log z\right)^{2A}.$$
So we can invoke the conjecture to find that
$$\sum_{j=1}^{\left(\log z\right)^{2A}}\mid \mathcal R_j\mid \geq \#\{g\mid J:\omega(g)=\lfloor \log z\rfloor\},$$
since each choice of $g$ must yield at least one $q$ for $j$ in this range.  Since $j<\frac z2$ and any prime divisor of $g$ is $\geq \frac z2$, we know that $(j,g)=1$ for any $g$.  So any prime $q$ can only appear in at most one set $\mathcal R_j$, and hence the $\mathcal R_j$ are pairwise disjoint.

Now, by the standard combinatorial identity that
\begin{gather}\label{combid}
\left(\begin{array}{c} n\\k\end{array}\right)\geq \left(\frac nk\right)^k,
\end{gather}
we know that
\begin{align*}
\#\{g\mid J:\omega(g)=\lfloor \log z\rfloor\}\geq &\left(\begin{array}{c}\lfloor \frac{z}{4\log z} \rfloor \\ \lfloor \log z\rfloor \end{array}\right)>\left(\frac{z}{5\log^2 z}\right)^{\log z-1}\\
\geq &\left(\frac{z}{5\log^2 z}\right)^{\log z}\left(\frac 1z\right)>\left(\frac{z}{15\log^2 z}\right)^{\log z},
\end{align*}
since $3^{\log z}>z$.
So there must exist a $j_0\leq \left(\log z\right)^{2A}$ such that
$$\mid \mathcal R_{j_0}\mid \geq \frac{\left(\frac{z}{15\log^2 z}\right)^{\log z}}{\left(\log z\right)^{2A}}.$$
Choose two disjoint subsets of $R_{j_0}$, each with $\left\lfloor\left(\frac{z}{16\log^2 z}\right)^{\log z}\right\rfloor$ elements.  We will call these subsets $\mathcal Q_1$ and $\mathcal Q_2$.  We then define
$$L_i=\prod_{q\in \mathcal Q_i}q.$$
For future notational ease, we note that
\begin{gather}\label{zeqn}
\left\lfloor\left(\frac{z}{16\log^2 z}\right)^{\log z}\right\rfloor=z^{\log z-(2+o(1))\log\log z}.\end{gather}
\section{The Sizes of $q$, $L_i$ and $\lambda(L_i)$}
Before we construct the sets $\mathcal P_{k_i}$, it will be useful to have information about the sizes of $q$, $L_i$, and $\lambda(L_i)$.  First, we find bounds for $q\in \mathcal Q_i$.
\begin{lemma}\label{qbound}
For any $q\in \mathcal Q_i$,
$$\left(\frac z6\right)^{\log z}\leq q\leq 2z^{\log z}\left(\log z\right)^{2A}.$$
\end{lemma}

\begin{proof}
For the upper bound, we use (\ref{gbound}) to find that
\begin{gather*}q=gj_0+1\leq 2gj_0\leq 2z^{\log z}j_0\leq 2z^{\log z}\left(\log z\right)^{2A}.\end{gather*}
For the lower bound, since $g$ has $\lfloor \log z\rfloor$ prime factors and each of the prime factors is $\geq \frac z2$,
\begin{gather*}q\geq \left(\frac z2\right)^{\log z-1}\geq \left(\frac z2\right)^{\log z}\left(\frac 1z\right)\geq \left(\frac z6\right)^{\log z},\end{gather*}
where again we use the fact that $3^{\log z}>z$.
\end{proof}
We use this to bound $L_i$.
\begin{lemma}\label{Lbound}
For $i=1$ or 2,
$$L_i\leq e^{\left(z^{\log z-(2+o(1))\log\log z}\right)\left(\log^2 z+2A\log\log z\right)}.$$
\end{lemma}
\begin{proof}
Using the upper bound for $q$ above as well as the size of $\mathcal Q_i$ given in (\ref{zeqn}), we see that
\begin{align*}L_i=&\prod_{q\in \mathcal Q_i}q\leq \left(2z^{\log z}\left(\log z\right)^{2A}\right)^{z^{\log z-(2+o(1))\log\log z}}\\
=&e^{\left(z^{\log z-(2+o(1))\log\log z}\right)\left(\log^2 z+2A\log\log z\right)},
\end{align*}
where the constant 2 at the front of the penultimate expression is absorbed onto the $o(1)$ term.
\end{proof}
Note that this implies
\begin{gather}\label{logL}\log(L_i)\leq z^{\frac 32\log z}.\end{gather}

By contrast, $\lambda(L)$ is quite a bit smaller.
\begin{lemma}\label{lambdaL} For $L_1$ and $L_2$ as above,
$$\lambda(L_1L_2)\leq e^{\frac 45z}.$$
\end{lemma}

\begin{proof}
For any prime $q\in \mathcal Q_i$, we know that $q-1\mid Jj_0$.  Since $$\lambda(L_1L_2)\mid LCM\left[q-1:q\in \mathcal Q_1\cup \mathcal Q_2\right],$$it follows that $\lambda(L_1L_2)\mid Jj_0$ as well.  We know that the number of primes between $\frac z2$ and $z$ is bounded loosely by $\frac{3z}{4\log z}$ (see, e.g., \cite{RS}), and hence
$$\lambda(L_1L_2)\leq Jj_0\leq z^{\frac{3z}{4\log z}}\left(\log z\right)^{2A}\leq z^{\frac{4z}{5\log z}}=e^{\frac 45z}.$$
\end{proof}

\section{The Set $\mathcal P_{k_1}$}
Next, we use $\mathcal Q_1$ and $L_1$ to construct one of the two sets of primes that will yield our Carmichael number.  For $k$ with $(k,\nu L_1L_2)=1$, define the set $\mathcal P_k^i$ by
$$\mathcal P_k^i=\{p=d_ik\nu+a:d_i\mid L_i,\omega(d_1)=z\},$$
where again $p$ denotes a (positive) prime.

We must now determine the size of $\mathcal P_{k}^i$ for our first choice of $k$.




\begin{lemma}\label{P1bound}
There exists a $k_1\leq 3\nu z^{A}\left(\log z\right)^{2A}$ such that
$$\mid \mathcal P_{k_1}^1\mid \geq z^{z\log z-(2+o(1))z\log\log z}.$$
\end{lemma}
\begin{proof}
Since we require $p=d_1k\nu+a$ and $(k,\nu)=1$, it is sufficient (though not necessary) to consider the congruence
\begin{gather}\label{modforp1}p\equiv a+d_1\nu\pmod{d_1\nu^2},\end{gather}
since we would then have $$p=d_1\nu(\nu k'+1)+a$$for some $k'$, and hence $k=\nu k'+1$ would be relatively prime to $\nu$.

Note that for any $d_1\mid L_1$, we can bound the modulus in (\ref{modforp1}) with
\begin{gather}\label{d1bound}k'<d_1\nu^2\leq \nu^2\left(2z^{\log z}\left(\log z\right)^{2A}\right)^{z}\leq \nu^2z^{z\log z+2Az\frac{\log\log z}{\log z}+z\frac{\log 2}{\log z}}.\end{gather}
Hence,
\begin{gather}\label{k1bound}\left(\log \left(d_1\nu^2\right)\right)^{A}\leq z^{A}\left(\log z\right)^A\left[\log z+3A\log\log z\right]^A<2 z^{A}\left(\log z\right)^{2A}.
\end{gather}
So we see as before that by the conjecture,
\begin{gather}\label{Pk1}\sum_{k'=1}^{2z^{A}\left(\log z\right)^{2A}}\mid \mathcal P_{\nu k'+1}^1\mid \geq \#\{d_1\mid L_1:\omega(d_1)=z\}.\end{gather}
If $z$ is sufficiently large relative to $|a|\nu$, we have
\begin{gather}\label{qk}
k=\nu k'+1\leq 3\nu z^{A}\left(\log z\right)^{2A}<\left(\frac z6\right)^{\log z}\leq q
\end{gather}
by Lemma \ref{qbound}.  So it follows that $(k,q)=1$ for every $q\mid L_1L_2$.  Thus, each $p$ appearing on the left-hand side of (\ref{Pk1}) appears exactly once.  Note that
$$\#\{d_1\mid L_1:\omega(d_1)=z\}\geq \left(\begin{array}{c}z^{\log z-(2+o(1))\log\log z} \\ z\end{array}\right)\geq z^{z\log z-(2+o(1))z\log\log z}.$$
by (\ref{combid}).  So there must exist a $k_1\leq 3\nu z^{A}\left(\log z\right)^{2A}$ such that
$$\mid \mathcal P_{k_1}^1\mid \geq \frac{z^{z\log z-(2+o(1))z\log\log z}}{3\nu z^{A}\left(\log z\right)^{2A}}=z^{z\log z-(2+o(1))z\log\log z}.$$
\end{proof}
For ease of notation, we will write $\mathcal P_{k_1}$ for $\mathcal P_{k_1}^1$.

\section{The Set $\mathcal P_{k_2}$}

Armed with this definition of $k_1$, we now define another set of primes $\mathcal P_{k_2}^2$, which we will denote $\mathcal P_{k_2}$.  The $k_2$ here will be chosen such that for any $p_1\in \mathcal P_{k_1}$ and $p_2\in \mathcal P_{k_2}$, we will have $(p_1-a,p_2-a)=\nu$ and $\nu\mid (p_i-a,P-a)$.  This is what will allow us to prove that $K=\nu$.

\begin{lemma}\label{P2bound}
There exists a $k_2\leq 7\nu^2 z^{2A}\left(\log z\right)^{4A}$ such that
$$\mid \mathcal P_{k_2}^2\mid \geq z^{z\log z-(2+o(1))z\log\log z}$$
and $(k_1,k_2)=1$.
\end{lemma}
\begin{proof}
Again, we choose a congruence condition that will be sufficient though not necessary:
\begin{gather*}
p\equiv a+d_2\nu \pmod{d_2\nu^2k_1}.
\end{gather*}
In this case, we have
$$p=d_2\nu(\nu k_1k'+1 )+a.$$
Letting $k=\nu k'k_1+1$, we see that $(k,k_1)=1$ and $(k,\nu)=1$.

Taking the log of the bound for $k_1$ in Lemma \ref{P1bound} gives
$$\log k_1\leq 3A\log z.$$
So we can use the bounds in (\ref{d1bound}) and Lemma \ref{P1bound} to find that
\begin{gather}\label{d2bound}d_2\nu^2k_1<\nu^2z^{z\log z+2Az\frac{\log\log z}{\log z}}\left(3\nu z^{A}\left(\log z\right)^{2A}\right)=z^{z\log z+(2A+o(1))z\frac{\log\log z}{\log z}},
\end{gather}
and hence
\begin{gather}\label{k2bound}\left(\log \left(d_2\nu^2k_1\right)\right)^{A}<\left(z\log^2 z+3Az\log \log z\right)^{A}<2z^{A}\left(\log z\right)^{2A}
\end{gather}
when $z$ is sufficiently large.  So as before,
$$\sum_{k'=1}^{2z^{A}\left(\log z\right)^{2A}}\mid \mathcal P_{\nu k'k_1+1}^2\mid \geq \#\{d_2\mid L_2:\omega(d_2)=z\}.$$
From here, the proof is similar to Lemma \ref{P1bound}, beginning with Equation (\ref{Pk1}).  We replace the bound for $k$ in (\ref{qk}) with
\begin{align*}
k=\nu k_1k'+1 & \leq 2\nu z^{A}\left(\log z\right)^{2A}k_1+1\\
&\leq 2\nu z^{A}\left(\log z\right)^{2A}\left(3\nu z^{A}\left(\log z\right)^{2A}\right)+1\\
&\leq 7\nu^2 z^{2A}\left(\log z\right)^{4A}.
\end{align*}
Clearly, this is still less than $\left(\frac z6\right)^{\log z}$, and hence the conclusion after (\ref{qk}) still applies.  Thus, there must exist a $k_2\leq 7\nu^2 z^{2A}\left(\log z\right)^{4A}$ such that $(k_1,k_2)=1$ and
$$\mid \mathcal P_{k_2}^2\mid \geq \frac{z^{z\log z-(2+o(1))z\log\log z}}{7\nu^2 z^{2A}\left(\log z\right)^{4A}}=z^{z\log z-(2+o(1))z\log\log z}.$$

\end{proof}
We now prove the claim that was made at the beginning of this section.

\begin{lemma}\label{K=nu}
Let $p_1\in \mathcal P_{k_1}$ and $p_2\in \mathcal P_{k_2}$.  Then $(p_1-a,p_2-a)=\nu$.
\end{lemma}

\begin{proof}
We have shown in Lemmas \ref{P1bound} and \ref{P2bound} that each $k_i$ is coprime to $\nu L_1L_2$ and that $(k_1,k_2)=1$.  Moreover, $(L_1,L_2)=1$, since the two numbers are comprised of nonintersecting sets of prime factors.  So $(L_1k_1\nu,L_2k_2\nu)=\nu$.  Since $\nu\mid p_1-1\mid L_1k_1\nu$ and $\nu\mid p_2-1\mid L_2k_2\nu$, we then have $(p_1-a,p_2-a)=\nu$.  This proves the lemma.
\end{proof}

\section{Bounding $P$}

In order to find a lower bound for $C_\nu(X,a)$, we will also need to define and bound $P$.  Here, we need only find the smallest prime $P$ such that
$$P=a\pmod{L_1L_2k_1k_2\nu}.$$
By the conjecture, there must then exist a $k_3$ for which
$$P=L_1L_2k_1k_2k_3\nu+a$$
with
$$k_3\ll \left(\log\left(L_1L_2k_1k_2\nu\right)\right)^A.$$
We can bound the size of $P$ with the following.
\begin{lemma}  The prime $P$ can be bounded by
$$P \leq e^{2\left(z^{\log z-(2+o(1))\log\log z}\right)\left(\log^2 z+2A\log\log z\right)}.$$
\end{lemma}

\begin{proof}
Recall from Lemma \ref{Lbound} that
$$L_i\leq e^{\left(z^{\log z-(2+o(1))\log\log z}\right)\left(\log^2 z+2A\log\log z\right)},$$
and from Lemma \ref{P2bound} that
$$k_i\ll 7\nu^2 z^{2A}\left(\log z\right)^{4A}.$$
So 
$$k_3\ll z^{2A\log z}.$$
Note that the $k_i$ can all be absorbed into the little-O term in the exponent of the $L_i$.  So
$$P \leq e^{2\left(z^{\log z-(2+o(1))\log\log z}\right)\left(\log^2 z+2A\log\log z\right)}.$$
\end{proof}

\section{Constructing a Carmichael Number}
Finally, we construct Carmichael numbers using these sets $\mathcal P_{k_1}$ and $\mathcal P_{k_2}$.  In order to do this, we recall a theorem of van Emde Boas and Kruyswijk \cite{EK} and Meshulam \cite{Me}.  Let $s(L)$ denote the smallest integer such that a sequence of at least $s(L)$ elements in $(\mathbb Z/L\mathbb Z)^\times$ must contain some nonempty sequence whose product is the identity.  Then we have the following.

\begin{theorem}\label{veb}
For any $L$,
$$s(L)<\lambda(L)\left(1+\log\left(\frac{\phi(L)}{\lambda(L)}\right)\right).$$
Moreover, let $v>t>s(L)$.  Then any sequence of $v$ elements in $(\mathbb Z/L\mathbb Z)^\times$ contains at least $\left(\begin{array}{c} v \\t\end{array}\right)/\left(\begin{array}{c} v \\s(L)\end{array}\right)$ distinct subsequences of length at least $t-s(L)$ and at most $t$ whose product is the identity.
\end{theorem}
In our case, we have the following bound for $s(L_1L_2k_1k_2k_3\nu)$.
\begin{lemma}\label{sbounds} For $L_i$, $k_i$, and $\nu$ as above,
$$s(L_1L_2k_1k_2k_3\nu)<e^{z}.$$
\end{lemma}
\begin{proof}
First,
$$\lambda(L_1L_2k_1k_2k_3\nu)\leq \lambda(L_1L_2)k_1k_2k_3\nu\leq e^{\frac 45z}\nu\left(21\nu^4z^{3A}\left(\log z\right)^{6A}\right)z^{2A\log z}\leq e^{\frac 56 z}$$
by Lemmas \ref{lambdaL}, \ref{P1bound}, and \ref{P2bound}.  Meanwhile, by (\ref{logL}),
$$\log(L_1L_2k_1k_2k_3\nu)\leq 2\log(L_1L_2)\leq 2z^{3\log z}=2e^{3\log^2 z}<e^{\frac 16 z}$$
when $z$ is large. Thus,
$$s(L_1L_2k_1k_2k_3\nu)<e^{z}.$$
\end{proof}

Now, for $i=1$ or 2, let $F_i(z,X)$ denote the set of square-free integers $n_i\leq X$ such that\\

\par (i) for any $p\mid n_i$, $p\in \mathcal P_{k_i}$, and

\par (ii) $n_i\equiv 1\pmod{L_1L_2k_1k_2k_3\nu}$.\\

Combining Theorem \ref{veb} and Lemma \ref{sbounds} gives the following.
\begin{lemma}\label{ni}
For $i=1$ or 2,
$$\left|F_i\left(z,z^{z^{z+1}\left(\log z+\left(2A+o(1)\right)\frac{\log\log z}{\log z}\right)}\right)\right| \geq z^{z^{z+1}\left(\log z-(2+o(1))\log\log z\right)}.$$


\end{lemma}
\begin{proof}
We prove this first for $i=2$; the case of $i=1$ can be proven with nearly identical reasoning but slightly better bounds.  

To this end, we recall that
$$\mid \mathcal P_{k_2}\mid \geq z^{z\log z-(2+o(1))z\log\log z}$$
by Lemma \ref{P1bound}.  Clearly, this is much larger than $s(L_1L_2k_1k_2k_3)$.  So define
\begin{align*}t=&z^z,\\ v=&z^{z\log z-(2+o(1))z\log\log z},\end{align*}
where $v$ is the greatest integer below the lower bound for $\mathcal P_{k_2}$ above.

We see that $t<v$.   Note in particular that 
\begin{gather*}\frac v{t}=z^{z\log z-(2+o(1))z\log\log z},\end{gather*}
since $t$ can be absorbed into the little-o term.  So
$$\left(\begin{array}{c} v \\ t\end{array}\right)\geq v^tt^{-t}=z^{t(z\log z-(2+o(1))z\log\log z)},$$
and
$$\left(\begin{array}{c} v\\ s(L) \end{array}\right)\leq v^{s(L)}\leq v^{e^z}=v^{o\left(\frac{t}{\log z\log \log z}\right)}=z^{o\left(\frac{tz}{\log z}\right)}.$$
So by Theorem \ref{veb}, the number of $n_2$ that can be constructed by products of at most $t$ elements and at least $t-s(L)$ elements in $\mathcal P_{k_2}$ is greater than or equal to
\begin{align*}  \left(\begin{array}{c} v \\ t\end{array}\right)/\left(\begin{array}{c} v\\ s(L) \end{array}\right)&\geq v^{t}t^{-t}v^{-e^z}\\
&=z^{t(z\log z-(2+o(1))z\log\log z)}\\
&=z^{z^{z+1}\left(\log z-(2+o(1))\log\log z\right)}.
\end{align*}
By (\ref{d2bound}) and (\ref{k2bound}), for any $p\in \mathcal P_{k_2}$,
$$p\leq z^{z\log z+\left(2A+o(1)\right)\frac{z\log\log z}{\log z}}.$$
Since any $n_2$ will have at most $t=z^z$ prime factors,
\begin{align*}
n_2\leq &\left(z^{z\log z+\left(2A+o(1)\right)\frac{z\log\log z}{\log z}}\right)^{z^{z}}\\
=&z^{z^{z+1}\left(\log z+\left(2A+o(1)\right)\frac{\log\log z}{\log z}\right)}.
\end{align*}
So
$$\left|F_2\left(z,z^{z^{z+1}\left(\log z+\left(2A+o(1)\right)\frac{\log\log z}{\log z}\right)}\right)\right| \geq z^{z^{z+1}\left(\log z-(2+o(1))\log\log z\right)}.$$
For the case of $i=1$, the proof is the same except that instead of Lemma \ref{P2bound} and Equations (\ref{d2bound}) and (\ref{k2bound}), we apply Lemma \ref{P1bound} and Equations (\ref{d1bound}) and (\ref{k1bound}).
\end{proof}
Finally, let

\begin{gather}\label{Xdef} 
X=X(z)=z^{2z^{z+1}\left(\log z+2A\frac{\log\log z}{\log z}\right)}e^{2z^{\log z}\left(\log^2 z+2A\log\log z\right)}.\end{gather}
The function from $z$ to $X(z)$ is well-defined, since it is monotone increasing on our range of $X$.

We give the following as a helpful lookup table comparing logs of $X$ to logs of $z$.
\begin{align*}
\log X=&2z^{z+1}\left(\log^2 z+\left(2A+o(1)\right)\log\log z\right),\\
\log\log X=&z\log z+O(\log z),\\
\log\log \log X=&(1+o(1))\log z,\\
\log\log \log \log X=&(1+o(1))\log\log z.
\end{align*}
This also means that
\begin{align*}
P&\leq e^{2z^{\log z}\left(\log^2 z+2A\log\log z\right)}\\
&\ll e^{\left(\log\log X\right)^{2\log\log\log X}}\\
&\ll X^{\frac{\left(\log\log X\right)^{2\log\log\log X}}{\log X}}\\
&\ll X^{\frac{1}{\left(\log X\right)^{1-\epsilon}}}
\end{align*}
for any $\epsilon>0$.

We can use Lemma \ref{ni} to prove Theorem \ref{MainTheorem}.  We will first prove Theorem \ref{MainTheorem} in the case where $X$ can be written in the specific form above.  
\begin{theorem}
Assume Conjecture 2 holds, and let $X$ be as defined as in (\ref{Xdef}) for a sufficiently large integer $z$.  Then for any $\nu$ with $(a,\nu)=1$ and $a$ and $\nu$ having opposite parity,
$$C_\nu(X,a)\geq X^{1-(2+o(1))\frac{\log\log \log \log X}{\log \log\log X}}.$$
\end{theorem}
\begin{proof}

From Lemma \ref{ni}, we can construct many $n_1$ and $n_2$ that are 1 modulo $L_1L_2k_1k_2\nu$.  So let $n=Pn_1n_2$.  By construction, we see that $$n\equiv a\pmod{L_1L_2k_1k_2k_3}.$$
Clearly, if $p\mid n$ then either $p\mid n_1$, in which case $p-a\mid \nu L_1k_1$, or $p\mid n_2$, in which case $p-a\mid \nu L_2k_2$, or else $p=P$ and hence $P-a=\nu L_1L_2k_1k_2k_3$.  In any of these cases, $p-a\mid \nu L_1L_2k_1k_2k_3\mid n-a$.  So $n$ is an $(a,a)$-Carmichael number.  Moreover, by Lemma \ref{K=nu}, we know that $K=\nu$ for this choice of $n$.

To find the number of such $n\leq X$, let
$$X'=z^{2z^{z+1}\left(\log z+2A\frac{\log\log z}{\log z}\right)}\leq \frac XP.$$
We recall that there are at least
$$z^{z^{z+1}\left(\log z-(2+o(1))\log\log z\right)}$$
choices for $n_1$ with $n_1\leq \sqrt{X'}$, and the same lower bound holds for the number of choices of $n_2$ with $n_2\leq \sqrt{X'}$.  So the number of $n=Pn_1n_2$ with $n\leq X$ is at least
$$z^{2z^{z+1}\left(\log z-(2+o(1))\log\log z\right)}.$$
This number can be rewritten as
\begin{align*}
z^{2z^{z+1}\left(\log z-(2+o(1))\log\log z\right)}=&z^{2z^{z+1}\left(\log z+2A\frac{\log \log z}{\log z}-(2+o(1))\log\log z\right)}\\
=&z^{-2z^{z+1}\left(2+o(1)\log\log z\right)}X^{1-O\left(\frac{1}{(\log X)^{1-\epsilon}}\right)}\\
=& \left(X^{-(2+o(1))\frac{\log\log z}{\log z+2A\frac{\log\log z}{\log z}}}X^{1-O\left(\frac{1}{(\log X)^{1-\epsilon}}\right)}\right)\\
=& \left(X^{-(2+o(1))\frac{\log\log z}{\log z}}X^{1-O\left(\frac{1}{(\log X)^{1-\epsilon}}\right)}\right).\\
\end{align*}
Recalling that $\log z=(1+o(1))\log\log\log X$ and $\log\log z=(1+o(1))\log\log\log\log X$, we can write the above as
$$=X^{1-(2+o(1))\frac{\log\log \log \log X}{\log \log\log X}}.$$
This proves the theorem.
\end{proof}

We note that this does not give full generality to choose any sufficiently large $X$, since it is required that $z$ be an integer.  We close this loophole now.
\begin{theorem}\label{ThmX0}
Assume Conjecture 2 holds, and let $X_0$ be sufficiently large.  Assume that $X_0$ can be written in the form of (\ref{Xdef}) but for a non-integer $z$.  Then for any $\nu$ with $(a,\nu)=1$ and $a$ and $\nu$ having opposite parity,
$$C_\nu(X_0,a)\geq X_0^{1-(2+o(1))\frac{\log\log \log \log X_0}{\log \log\log X_0}}.$$
\end{theorem}
\begin{proof}
Let $y$ be the integer such that $X(y)<X_0<X(y+1)$.  Since $X(z)$ is a monotone increasing function with an increasing first derivative when $z$ is sufficiently large,
$$X_0-X(y)\leq X(y+1)-X(y)\leq \frac{dX(y+1)}{dy}\leq y^{y^2}X(y).$$
So
$$\log X_0=\log X(y)(1+o(1)).$$
So we can use our log lookup table above with $X_0$ replacing $X$, finding
$$X_0\leq (\log X_0)^{(1+o(1))\log\log X_0}X(y),$$
which can be rewritten as
$$X_0^{1-\frac{(1+o(1))(\log \log X_0)^2}{\log X_0}}\leq X(y).$$
So
\begin{align*}C_\nu(X_0,a)&\geq C_\nu(X(y),a)\\
&\geq (X(y))^{1-(2+o(1))\frac{\log\log \log \log X(y)}{\log \log\log X(y)}}\\
&\geq X_0^{1-(2+o(1))\frac{\log\log \log \log X_0}{\log \log\log X_0}-\frac{(1+o(1))(\log \log X_0)^2}{\log X_0}}.
\end{align*}
Since
$$\frac{(\log \log X_0)^2}{\log X_0}=o\left(\frac{\log\log \log \log X_0}{\log \log\log X_0}\right),$$
the inequality can be simplified to 
$$C_\nu(X_0,a)\geq X_0^{1-(2+o(1))\frac{\log\log \log \log X_0}{\log \log\log X_0}},$$
thereby proving Theorem \ref{ThmX0}, and hence proving Theorem \ref{MainTheorem}.
\end{proof}

\vskip20pt\noindent {\bf Acknowledgements.} We would like to thank Jonathan Webster for asking a question that prompted the writing of this paper.  We also wish to thank anonymous referees for some very helpful suggestions.  Additionally, we would like to thank Carl Pomerance for some very helpful feedback, and Daniel Larsen for spurring us to prove this result in greater generality.  Finally, we are also grateful for a Wofford College Summer Grant that funded this work.

\bibliographystyle{line}

\end{document}